\documentclass{article}

\usepackage{amsmath}
\usepackage{amsfonts}
\usepackage{amssymb}
\usepackage{amsthm}
\usepackage{array}
\usepackage{chngpage}
\usepackage{comment}
\usepackage{float}
\usepackage[OT2,T1]{fontenc}
\usepackage{graphicx,caption}
\usepackage[export]{adjustbox}
\usepackage{longtable}
\usepackage{mathrsfs}
\usepackage{wrapfig}
\usepackage{cutwin}
\usepackage{shapepar}
\usepackage{tikz}
\usepackage[lowtilde]{url}
\usetikzlibrary{matrix,arrows}

\newcommand{\C}{\mathbb{C}}

\newcommand{\N}{\mathbb{N}}
\newcommand{\Q}{\mathbb{Q}}
\newcommand{\R}{\mathbb{R}}
\newcommand{\Z}{\mathbb{Z}}

\newcommand{\cO}{\mathcal{O}}

\DeclareSymbolFont{cyrletters}{OT2}{wncyr}{m}{n}
\DeclareMathSymbol{\sha}{\mathalpha}{cyrletters}{"58}

\newcommand{\eps}{\varepsilon}

\newlength{\strutheight}
\newtheorem{theorem}{Theorem}[section]
\newtheorem{lemma}[theorem]{Lemma}
\newtheorem{corollary}[theorem]{Corollary}
\newtheorem{figurecap}[theorem]{Figure}

\author{Alex Cowan\\cowan@math.harvard.edu}
\title{The distribution of multiples of real points on an elliptic curve.}
\date{}

\begin{document}
\maketitle
\begin{abstract}
\noindent Given an elliptic curve $E$ and a point $P$ in $E(\R)$, we investigate the distribution of the points $nP$ as $n$ varies over the integers, giving bounds on the $x$ and $y$ coordinates of $nP$ and determining the natural density of integers $n$ for which $nP$ lies in an arbitrary open subset of $\R^2$. Our proofs rely on a connection to classical topics in the theory of Diophantine approximation.
\end{abstract}
\section{Introduction}
\noindent
Let $E: y^2 = 4x^3 - g_2x - g_3$ be an elliptic curve with $g_2,g_3 \in \R$, and suppose that $P$ is an element of infinite order in the group $E(\R)$. In this paper we investigate the statistics of the coordinates $x(nP), y(nP) \in \R$ of $nP$ for $n \in \Z$. The set of points $(x,y) \in \R^2$ which satisfy the equation for $E$ form either one or two connected subsets of $\R^2$, depending on whether the polynomial $4x^3 - g_2x - g_3$ has one or three real roots. In the case where $4x^3 - g_2x - g_3$ has three real roots, the coordinates of points making up one of the connected subsets are bounded, while in the other the coordinates are unbounded. In this case we will say that $E(\R)$ has two connected components, and we will refer to them as the ``bounded component'' and ``unbounded component''. If instead $4x^3 - g_2x - g_3$ has only one real root, then we will say that $E(\R)$ has only one component, we will refer to it as the ``unbounded component''.\\
\\
Let $\omega$ be the holomorphic differential $\frac{dx}{y}$ on $E$. We will say that the \textit{periods of $E$} are any two complex numbers $\omega_1$ and $\omega_2$ with the property that for any closed loop $C$ in $\C$, there exist integers $m$ and $n$ such that $\int_{C}\omega = m\omega_1 + n\omega_2$. As described in \cite{silverman1}, there are contours $C_1$ and $C_2$ which enclose exactly two of the three roots of $4x^3 - g_2x - g_3$ such that $\omega_1 = \int_{C_1}\omega$ and $\omega_2 = \int_{C_2}\omega$. Moreover, it is always possible for $\omega_1$ and $\omega_2$ to be chosen such that $\omega_1 \in \R_{>0}$, $\text{Im}(\omega_2) > 0$, and $\text{Re}(\omega_2) = 0$ (if $E$ has two connected components) or $\frac{1}{2}\omega_1$ (if $E$ has only one connected component), as described in algorithm 7.4.7 of \cite{cohen}.\\
\\
In section \ref{growthsection}, we prove theorems which explain how large the coordinates of $nP$ get as a function of $n$:
\begin{figure}[H]
\begin{minipage}[t]{0.6\linewidth}
\begin{adjustwidth}{0em}{1em}
\begin{theorem}\label{absbound}
Suppose that $E/\C$ has periods $\omega_1$ and $\omega_2$, chosen such that $\omega_1 \in \R_{>0}$ and $\text{\emph{Im}}(\omega_2) > 0$. Then for every point $P$ of infinite order in the unbounded component of $E(\R)$, there exist infinitely many $n$ such that
  $$x(nP) > \frac{5}{\omega_1^2}n^2 + \cO(n^{-2}) \quad\quad\text{and}\quad\quad y(nP) > \frac{2\cdot 5^{\frac{3}{2}}}{\omega_1^3}n^3 + \cO(n^{-1}).$$
  If $P$ is instead a point of infinite order on the bounded component of $E(\R)$ (in the case where $E(\R)$ has two connected components), then there exist infinitely many $n$ such that
  $$x(nP) > \frac{5}{4\omega_1^2}n^2 + \cO(n^{-2}) \quad\quad\text{and}\quad\quad y(nP) > \frac{5^{\frac{3}{2}}}{4\omega_1^3}n^3 + \cO(n^{-1}).$$
  The implied constants depend only on $E$.\\
\end{theorem}
\end{adjustwidth}
\end{minipage}
\begin{minipage}[t]{0.4\linewidth}
\centering
\includegraphics[
                 width=0.95\linewidth, height = 5.7cm, valign=t,right]{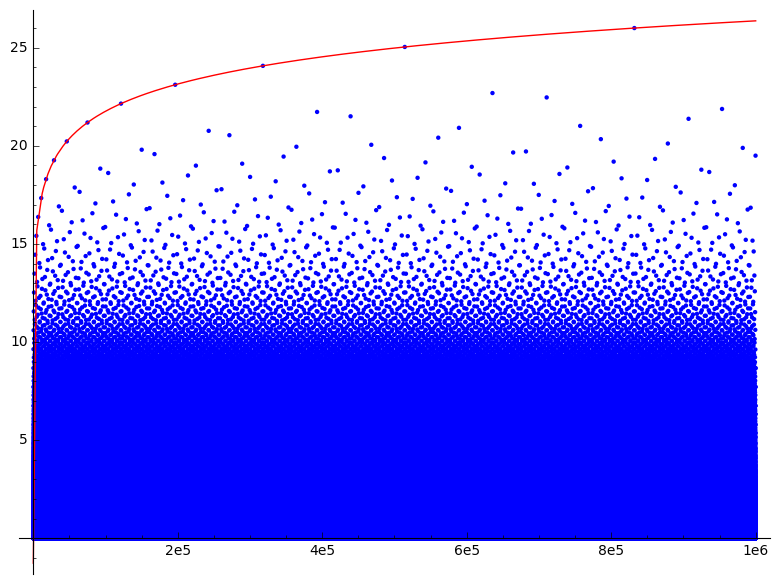}
\begin{adjustwidth}{1em}{0em}
  {\small \begin{figurecap} $\{\log(x(nP)+2)\,:\,1 < n < 10^6\}$ for $P \approx (-0.406,0.966)$ on $E: y^2 = x^3 + 1$, with the lower bound of theorem \ref{absbound} in red.\end{figurecap}}
\end{adjustwidth}
\end{minipage}%
\end{figure}

\begin{theorem}\label{typical}
Let $\psi$ be a non-decreasing function from $\N$ to $\R_{>0}$. If $\sum\limits_{n=1}^\infty \psi(n)^{-1}$ diverges, then for all points $P$ in $E(\R)$ except for a set of points of Lebesgue measure zero, there exist infinitely many positive integers $n$ such that
  $$x(nP) > \psi(n)^2\quad\quad\text{and}\quad\quad y(nP) > \psi(n)^3,$$
  while if $\sum\limits_{n=1}^\infty \psi(n)^{-1}$ converges, then the set of points $P$ in $E(\R)$ for which there exist infinitely many such $n$ has measure zero.
\end{theorem}
\begin{theorem}\label{arbitrarilybig}
  For any $E$ and any function $\psi:\N\to\R_{>0}$, there exists a point $P$ in $E(\R)$ such that, for infinitely many positive integers $n$,
  $$x(nP) > \psi(n)^2\quad\quad\text{and}\quad\quad y(nP) > \psi(n)^3.$$
\end{theorem}
~\\
Variants of these theorems can be given for general $P \in E(\C)$, and not just for $P \in E(\R)$. For example,
\begin{theorem}\label{complexdirichlet}
  Let $P$ be a point in $E(\C)$ of infinite order. Then
  $$|x(nP)| \gg n \quad\quad\text{and}\quad\quad |y(nP)| \gg n^{\frac{3}{2}},$$
  where the implied constants depends only on $E$.
\end{theorem}
~\\
The proofs of these theorems rely on the work of Hurwitz \cite{hur1891}, Khinchin \cite{khin64} \cite{khin26}, and Dirichlet (see \cite{hw08}, theorem 200) in the field of Diophantine approximation. The correspondence between results in Diophantine approximation and asymptotics for the size of the coordinates of $nP$ can be extended further.\\
~\\
In section \ref{distsection}, we investigate the full distribution of the $x$ and $y$ coordinates of $nP$. Let $\omega_1$ and $\omega_2$ be the periods of $E/\C$, chosen such that $\omega_1 \in \R_{>0}$ and $\text{Im}(\omega_2) > 0$. Let $\Lambda$ be the lattice in $\C$ with basis $\langle \omega_1, \omega_2\rangle$. Then $E/\C$ is parameterized by elements $z$ of $\C/\Lambda$ via $z \mapsto (\wp(z),\wp'(z))$, where
$$\wp(z) := \frac{1}{z^2} + \sum_{\substack{\lambda \in \Lambda\\\lambda\neq 0}}\left(\frac{1}{(z-\lambda)^2} - \frac{1}{\lambda^2}\right)$$
and $\wp'$ is the derivative $\frac{d\wp}{dz}$. We prove the following regarding the distribution of integer multiples of a fixed $P \in E(\R)$ in section \ref{distsection}, which states essentially that these integer multiples of $P$ are ``equidistributed'' in a sense which is clarified in section \ref{distsection}.

\begin{theorem}\label{dist}
  Let $P$ be a point of infinite order in $E(\R)$, and let $z_P$ be the preimage of $P$ under the parameterization $z \mapsto (\wp(z),\wp'(z))$. Let $\omega_1$ and $\omega_2$ be the periods of $E/\C$, chosen such that $\omega_1 \in \R_{>0}$ and $\text{\emph{Im}}(\omega_2) > 0$. Let $\Lambda$ be the lattice in $\C$ with basis $\langle \omega_1, \omega_2\rangle$. Define $I_P \subset \C/\Lambda$ as follows:
  \begin{align*}
    I_P := \begin{cases}[0,\omega_1], &\text{\emph{Im}}(z_P) = 0 \mod\Lambda,\\
      [0,\omega_1]\cup \left([0,\omega_1]+\frac{\omega_2}{2}\right), &\text{\emph{Im}}(z_P)=\text{\emph{Im}}\!\left(\frac{\omega_2}{2}\right) \mod\Lambda,\end{cases}
  \end{align*}
  where $[0,\omega_1]$ denotes the interval of real numbers. Then, for any $U \subseteq \R^2$, we have
  $$\lim_{n\to\infty}\frac{1}{2n}\#\{|k|<n\,:\,(x(kP),y(kP)) \in U\} = \frac{\mu\!\left(\{z\in I_P\,:\,(\wp(z),\wp'(z)) \in U\}\right)}{\mu\!\left(I_P\right)},$$
  where $\mu$ is the Lebesgue measure.
\end{theorem}
\noindent
\begin{figure}[H]
\begin{minipage}[t]{0.4\linewidth}
\centering
\includegraphics[
                 width=0.9\linewidth, valign=t,left]{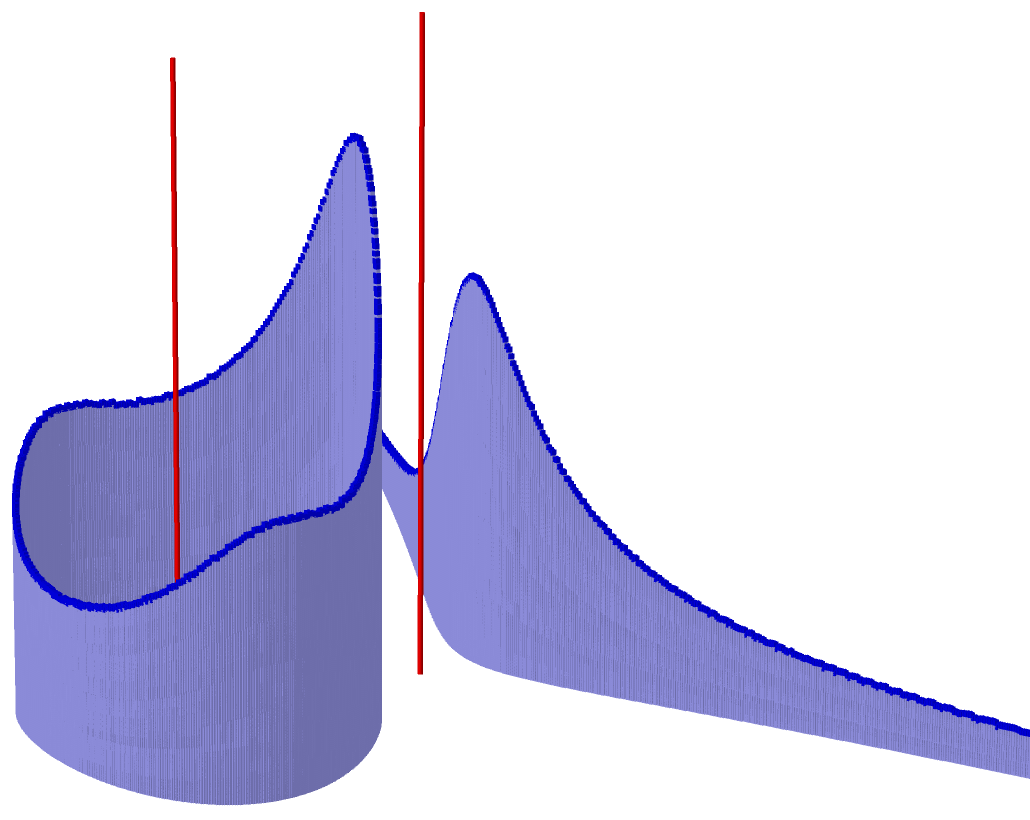}
\end{minipage}%
\begin{minipage}[t]{0.6\linewidth}
\begin{corollary}\label{density} Fix $P_0 = (x_0,y_0) \in E(\R)$ and $\eps > 0$. For all $P \in E(\R)$ of infinite order, the natural density of integers $n$ for which $(x(nP) - x_0)^2 + (y(nP) - y_0)^2 < \eps^2$ is
  $$\frac{2\eta(\eps + \cO(\eps^2))}{\omega_1\sqrt{y_0^2 + \left(6x_0^2 - \frac{g_2}{2}\right)^2}},$$
  where $\eta = 1$ if both $P$ and $P_0$ are on the unbounded component of $E(\R)$, $\eta = \tfrac{1}{2}$ if $P$ is on the bounded component of $E(\R)$, and $\eta = 0$ if $P_0$ is on the bounded component of $E(\R)$ but $P$ is not. The implied constant depends only on $E$ and $P_0$.
\end{corollary}
\end{minipage}
\end{figure}
\noindent{\small \begin{figurecap}\label{densfig} $\{nP\,:\,|n| < 5\cdot 10^5,\, x(nP) < 1.89\}$ for $P = (0,0)$ on \emph{E37a}: $y^2 + y = x^3 - x$ \cite{lmfdb37a}, with the poles of the density function from corollary \ref{density} shown.\end{figurecap}}
\noindent We then obtain the following spacing law:
\begin{figure}[H]
  \begin{minipage}[t]{0.6\linewidth}
    \begin{corollary}\label{spacing} Let $E:y^2 = x^3 + ax + b$ be an elliptic curve, let $Q = (x_Q,y_Q)$ be an arbitrary fixed point in $E(\R)$, and let $d$ be an arbitrary real number. Define
  \begin{align*}
    &F_{\pm,Q}(x) := \left(\frac{\pm\sqrt{x^3 + ax + b} - y_Q}{x - x_Q}\right)^2 - 2x - x_Q\\
    &\text{\it{and}}\\
    &\rho(x) := \frac{1}{\sqrt{x^3 + ax + b}}.
  \end{align*}
  \textit{Let $x^\pm_1,\ldots,x^\pm_{k^\pm}$ be the real solutions to $F_{\pm,Q}(x) = d$. Then, for any point $P$ in $E(\R)$ of infinite order, the distribution of the values $x(nP + Q) - x(nP)$ as $n$ varies over the integers is proportional to the function $f(d)$, defined as
  $$f(d) := \sum_{i = 1}^{k^+}\!{\vphantom{\sum}}^*\frac{\rho(x^+_i)}{F'_{+,Q}(x^+_i)} + \sum_{i = 1}^{k^-}\!{\vphantom{\sum}}^*\frac{\rho(x^-_i)}{F'_{-,Q}(x^-_i)},$$
  where $\sum^*$ indicates that, if $P$ is on the unbounded component of $E(\R)$, then the sum omits the $x^\pm_i$ for which $x^\pm_i$ is not the $x$-coordinate of any point on the unbounded component of $E(\R)$.}\end{corollary}
\end{minipage}
\begin{minipage}[t]{0.4\linewidth}
  \vspace*{-1cm}
  \hspace*{0cm}
\centering
\includegraphics[width=0.95\linewidth,height=3.8cm, valign=t,right]{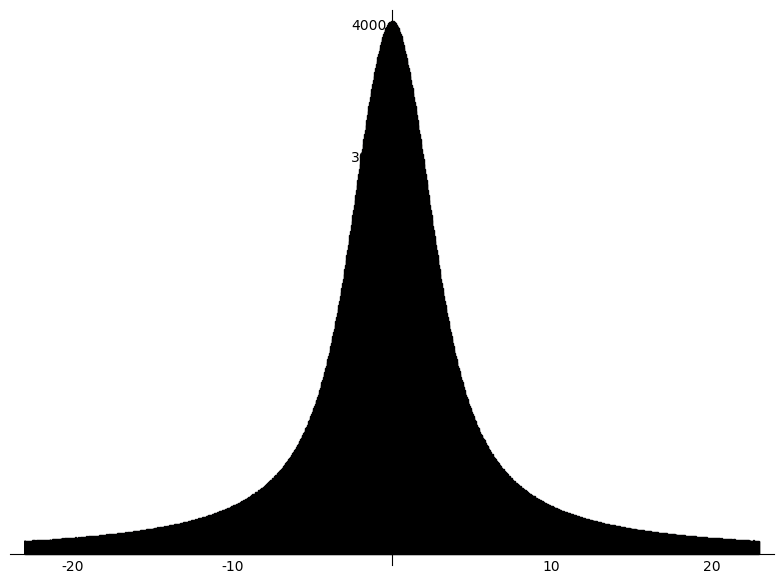}
\begin{adjustwidth}{1em}{0.2em}
  {\small \begin{figurecap} $\{x((n+1)P) - x(nP)\,:\,1 < n < 10^6\}$ for $P \approx (-0.406,0.966)$ on $E: y^2 = x^3 + 1$, with the top 10\% and bottom 10\% of the data omitted.\end{figurecap}}
    \end{adjustwidth}
\includegraphics[width=0.95\linewidth,height=3.8cm, valign=t,right]{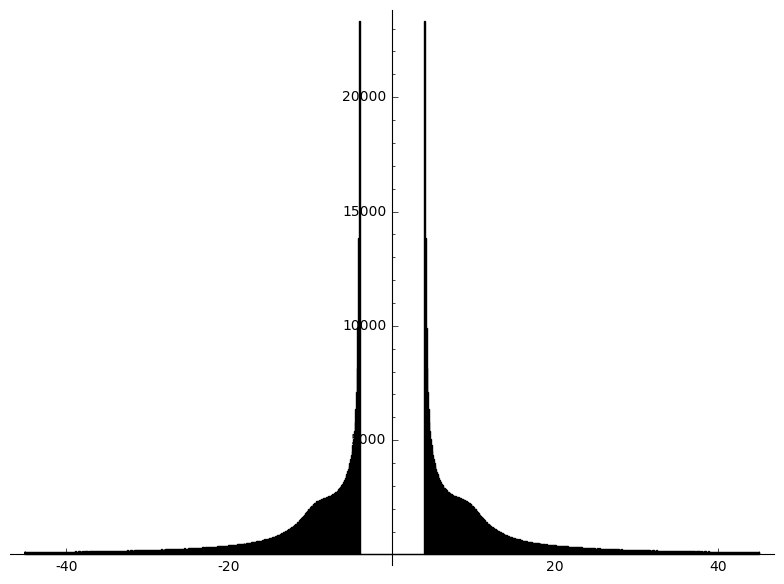}
    \begin{adjustwidth}{1em}{0.2em}
      {\small \begin{figurecap} $\{x((n+1)P) - x(nP)\,:\,1 < n < 10^6\}$ for
          $P = (0,4)$ on E37a: $y^2 = x^3 - 16x + 16$, with the top 10\% and bottom 10\% of the data omitted.\end{figurecap}}
    \end{adjustwidth}

\end{minipage}%
\end{figure}
~\\
We also show in corollaries \ref{spacingmomentslower} and \ref{spacingmomentasymptotic} that the raw moments of the function $f(d)$ diverge, and give an upper bound for the associated partial sums.\\
~\\
As an application of these growth and distribution results, we explain certain numerical observations of Bremner and Macleod made in \cite{bm14}. There, for every integer $N \leq 1000$, Bremner and Macleod find the positive integer solutions $a,b,c$ to the equation
\begin{equation}\label{fruit}
  \frac{a}{b+c}+\frac{b}{a+c}+\frac{c}{a+b} = N.
\end{equation}
Solutions to (\ref{fruit}) are given by certain rational points on certain elliptic curves $E_N$. If $E_N$ has rank $1$ and $P$ is a generator for $E_N$, then Bremner and Macleod make numerical observations regarding the set of $n \in \Z$ for which $nP$ yields a solution to equation (\ref{fruit}). In particular, they tabulate the smallest positive integers $n$ which yield solutions as $N$ varies, and ask about the proportion of integers $n$ that yield solutions for fixed $N$. Using corollary \ref{density} we give the proportion exactly.\\
\section{Background}\label{background}
Let $E/\C$ be an elliptic curve with periods $\omega_1$ and $\omega_2$, chosen such that $\omega_1 \in \R_{>0}$ and $\text{Im}(\omega_2) > 0$, and let $\Lambda$ be the lattice in $\C$ with basis $\langle \omega_1, \omega_2 \rangle$. As stated in the introduction, the Weierstrass-$\wp$ function defined by
$$\wp(z) := \frac{1}{z^2} + \sum_{\substack{\lambda \in \Lambda\\\lambda\neq 0}}\left(\frac{1}{(z-\lambda)^2} - \frac{1}{\lambda^2}\right)$$
gives a parameterization $\C/\Lambda \to E(\C)$ via $z \mapsto (\wp(z),\wp'(z))$. This parameterization is discussed at length in \cite{silverman1}, \cite{silverman2}, and \cite{darmon}, for example.\\
\\
The function $\wp(z)$ has a pole of order $2$ when $z \in \Lambda$ and has no other poles. From this it follows that the set of $z \in \C/\Lambda$ which have imaginary part $0$ modulo $\Lambda$ maps to the unbounded component of $E(\R)$, and, if $\omega_2$ can be chosen to be purely imaginary, the set of $z \in \C/\Lambda$ which have imaginary part $\tfrac{1}{2}\text{Im}(\omega_2)$ modulo $\Lambda$ maps to the bounded component of $E(\R)$. If $z_P$ has imaginary part $\tfrac{1}{2}\text{Im}(\omega_2)$ modulo $\Lambda$, then $nz_P$ will as well exactly when $n$ is odd, and thus the $x$ and $y$ coordinates of $nP$ will be on the bounded component of $E(\R)$ exactly when $n$ is odd. Moreover, we will use the fact that $\wp$ is an even function, and that the Laurent expansion of $\wp$ at its poles has no constant term.\\
\\
For a real number $r$, let $\{r\}$ denote the distance from $r$ to the nearest integer. If $P \in E(\R)$ and $x(nP)$ and $y(nP)$ denote the $x$ and $y$ coordinates of $nP$ respectively, then it follows from the previously stated facts about the Laurent expansion of $\wp$ at its poles that
$$x(nP) = \omega_1^{-2}\left\{n\frac{z_P}{\omega_1}\right\}^{-2} + \cO\!\left(\left\{n\frac{z_P}{\omega_1}\right\}^{2}\right) \quad\quad\text{and}\quad\quad y(nP) = -2\omega_1^{-3}\left\{n\frac{z_P}{\omega_1}\right\}^{-3} + \cO\!\left(\left\{n\frac{z_P}{\omega_1}\right\}\right)$$
when $z_P$ is the element of $\C/\Lambda$ which maps to $P$ under the parameterization $z \mapsto (\wp(z),\wp'(z))$, and the implied constants depend only on $E$.\\
\\
The following lemmas will be useful for studying $\left\{n\frac{z_P}{\omega_1}\right\}$:\\
\begin{lemma}\label{hurwitz} \textbf{\emph{(Hurwitz)}} For all irrational $\alpha$ there exist infinitely many natural numbers $n$ such that
  $$\{n\alpha\} < \frac{1}{\sqrt{5}n}.$$
\end{lemma}
\begin{proof} See \cite{hur1891}, Satz 1.\end{proof}
\begin{lemma}\label{dirichletsimul} \textbf{\emph{(Dirichlet)}} Let $\alpha_1,\dots,\alpha_k$ be irrational numbers. There exist infinitely many natural numbers $n$ such that
  $$\{n\alpha_j\} < \frac{1}{n^{\frac{1}{k}}}$$
  for all $j \in \{1,2,\dots,k\}$.
\end{lemma}
\begin{proof} See \cite{hw08} theorem 200.\end{proof}
\begin{lemma}\label{khinchin} \textbf{\emph{(Khinchin)}} Let $\psi$ be a non-decreasing function from $\N$ to $\R_{>0}$, If $\sum\limits_{n=1}^\infty \psi(n)^{-1}$ diverges, then for all real numbers $\alpha$ except for a set of Lebesgue measure zero, there exist infinitely many natural numbers $n$ such that
  $$\{n\alpha\} < \frac{1}{\psi(n)},$$
  while if $\sum\limits_{n=1}^\infty \psi(n)^{-1}$ converges, then the set of real numbers $\alpha$ for which there exist infinitely many natural numbers $n$ such that
  $$\{n\alpha\} < \frac{1}{\psi(n)}$$
  has Lebesgue measure zero.
\end{lemma}
\begin{proof} See \cite{khin26}. \end{proof}
~\\
In the opposite direction, we have the following:
\begin{lemma}\label{conv} For any function $\psi(n):\N\to\R_{>0}$, there exists a real number $\alpha$ such that the inequality
  $$\{n\alpha\} < \frac{1}{\psi(n)}$$
  is satisfied for infinitely many natural numbers $n$.
\end{lemma}
\begin{proof} See \cite{khin64}, theorem 22.\end{proof}
\section{Growth Rates}\label{growthsection}
Using lemmas \ref{hurwitz}, \ref{dirichletsimul}, \ref{khinchin}, and \ref{conv} we can now prove theorems \ref{absbound}, \ref{complexdirichlet}, \ref{typical}, and \ref{arbitrarilybig}.\\
\begin{proof}[Proof of theorem \ref{absbound}] First suppose that $P$ is a point of infinite order on the unbounded component of $E(\R)$, and let $z_P$ be the preimage of $P$ under the parameterization $\C/\Lambda \to E(\C)$ defined by $z \mapsto (\wp(z),\wp'(z))$, where $\Lambda$ is the lattice in $\C$ with basis $\langle \omega_1, \omega_2 \rangle$. Then $z_P$ is real modulo $\Lambda$. From the observations in section \ref{background}, we have
$$\wp(nz_P) = \omega_1^{-2}\left\{n\frac{z_P}{\omega_1}\right\}^{-2} + \cO\!\left(\left\{n\frac{z_P}{\omega_1}\right\}^2\right) \quad\text{and}\quad \wp'(nz_P) = -2\omega_1^{-3}\left\{n\frac{z_P}{\omega_1}\right\}^{-3} + \cO\!\left(\left\{n\frac{z_P}{\omega_1}\right\}\right),$$
where the implied constants depend only on $E$. Lemma \ref{hurwitz} implies that the inequality $\left\{n\frac{z_P}{\omega_1}\right\} < \frac{1}{\sqrt{5}n}$ holds for infinitely many $n$, so for these $n$ we have
$$\wp(nz_P) > \frac{5}{\omega_1^2}n^2 + \cO(n^{-2}),$$
and
$$\wp'(nz_P) > \frac{2\cdot 5^{\frac{3}{2}}}{\omega_1^3}n^3 + \cO(n^{-1}).$$
Now if instead $P$ is on the bounded component of $E(\R)$, then $2z_P$ is real modulo $\Lambda$, so the argument above can be applied to $2P$.\end{proof}
~\\
Repeating this argument and using lemma \ref{dirichletsimul} in the case where $k = 2$, $\alpha_1 = \text{Re}(z_P)$, and $\alpha_2 = \text{Im}(z_P)$ proves theorem \ref{complexdirichlet}. Repeating the argument and using lemma \ref{khinchin} instead of lemma \ref{hurwitz} proves theorem \ref{typical}. Finally, using lemma \ref{conv} in this argument proves theorem \ref{arbitrarilybig}.\\
\section{Distributions}\label{distsection}
Next we turn our attention to results about the full distribution of $x(nP)$ and $y(nP)$ as $n$ varies. Let $X$ be a topological space with a measure $\mu$. We say that a sequence $(a_n)$ of elements of $X$ is equidistributed with respect to $\mu$ if and only if, for every function $f:X\to\C$, we have
$$\lim_{n\to\infty}\frac{1}{n}\sum_{k=1}^n f(a_k) = \int_X f(x)\,d\mu.$$
Sometimes we will say that a sequence is equidistributed in a space if it's clear what the associated measure is. In particular, we will say that a sequence is \textit{equidistributed modulo 1} if and only if it is equidistributed in the interval $[0,1]$ with respect to the Lebesgue measure.\\
\\
The following result, due to Weyl \cite{weyl16}, is an important tool for proving that certain sequences are equidistributed modulo $1$:
\begin{lemma}\label{weyl}\textbf{\emph{(Weyl's criterion)}} A sequence $(a_n)$ of real numbers is equidistributed modulo $1$ if and only if for every nonzero integer $\ell$ we have
  $$\lim_{n\to\infty}\frac{1}{n}\sum_{k=1}^n\exp(2\pi i \ell a_k) = 0$$
\end{lemma}
~\\
This lemma implies that, for any irrational $\alpha$, the sequence $a_n = \{n\alpha\}$ is equidistributed modulo $1$ (\cite{weyl16}, Satz 2). Theorem \ref{dist} is an immediate consequence of this fact.\\
\\
One simple application of theorem \ref{dist} comes from taking $U = [X,\infty]\times\R$ for large $X$ or $U = \R\times[Y,\infty]$ for large $Y$. Then, if $P$ is on the unbounded component of $E(\R)$, we have
\begin{align*}
  \lim_{n\to\infty}\frac{1}{2n}\#\{|k|<n\,:\,x(kP) > X\} &= \frac{\mu(\{z\in [0,\omega_1]\,:\,\wp(z) > X\})}{\mu([0,\omega_1])}\\
  &= \frac{\mu\!\left(\left\{z\in \left[-\frac{\omega_1}{2},\frac{\omega_1}{2}\right]\,:\,\frac{1}{z^2} + \cO(z^2) > X\right\}\right)}{\omega_1}\\
  &= \frac{2}{\omega_1}X^{-\frac{1}{2}} + \cO(X^{-\frac{5}{2}})\\
\intertext{(where we have used the periodicity of $\wp$ modulo $\Lambda$ and the facts about the Laurent expansion of $\wp$ around $0$ which were mentioned in section \ref{background}), and similarly}
  \lim_{n\to\infty}\frac{1}{2n}\#\{|k|<n\,:\,y(kP) > Y\} &= \frac{2^{\frac{1}{3}}}{\omega_1}Y^{-\frac{1}{3}} + \cO(Y^{-\frac{5}{3}}),
\end{align*}
~\\
where the implied constants depend only on $E$.\\
\\
\\
We now move to discussion of corollary \ref{density}. This corollary is useful because it gives a description of the distribution of $nP$ in a way which does not depend on knowledge of the function $\wp$.
\begin{proof}[Proof of corollary \ref{density}] Let $z_0$ be the preimage of $P_0$ under the map $z \mapsto (\wp(z),\wp'(z))$, and let $\Delta u$ be a real number. Then
$$(\wp(z_0 + \Delta u), \wp'(z_0 + \Delta u)) = \Big(\wp(z_0) + \wp'(z_0)\Delta u + \cO((\Delta u)^2),\;\; \wp'(z_0) + \wp''(z_0)\Delta u + \cO((\Delta u)^2)\Big).$$
Here the implied constant depends on both $E$ and $P_0$. From the equation of $E$ we can deduce that $\wp''(z_0) = 6\wp(z_0)^2 - \frac{g_2}{2}$. Using this fact and the preceding equation, we see that the point $(\wp(z_0 + \Delta u), \wp'(z_0 + \Delta u))$ will satisfy $(\wp(z_0+\Delta u) - x_0)^2 + (\wp'(z_0 + \Delta u) - y_0)^2 < \eps^2$ if and only if
$$|\Delta u| < \frac{\eps}{\sqrt{y_0^2 + \left(6x_0^2 - \frac{g_2}{2}\right)^2}}\left(1 - (\wp'(z_0) + \wp''(z_0))\frac{\cO((\Delta u)^3)}{\eps^2}\right).$$
Then, after using theorem \ref{dist} and noting that $\cO((\Delta u)^3)\eps^{-2} = \cO(\eps)$, we can conclude the result.\end{proof}
~\\
The following variation of corollary \ref{density} will simplify calculations in sections \ref{spacingsection} and \ref{bmsection}.
\begin{lemma}\label{diffdensity}
  Let $I$ be an open interval of length $\eps$ which is contained in the set of $x$-coordinates of points in $E(\R)$. For any $P \in E(\R)$ of infinite order, the natural density of integers $n$ for which $x(nP) \in I$ is proportional to
  $$\eta \cdot \frac{1}{\sqrt{4x^3 - g_2 x - g_3}}\cdot \eps + \cO(\eps^2),$$
  where $\eta = 1$ if $P$ is on the unbounded component of $E(\R)$ and $I$ is contained in the set of $x$-coordinates of points on the unbounded component of $E(\R)$, $\eta = \frac{1}{2}$ if $P$ is on the bounded component of $E(\R)$, and $\eta = 0$ otherwise. The implied constant depends only on $E$ and $I$.
\end{lemma}
\begin{proof}
  Write $I = \left(x_0 - \frac{\eps}{2}, x_0 + \frac{\eps}{2}\right)$. Corollary \ref{density} implies that the natural density of integers $n$ for which $x(nP) \in I$ is proportional to
  $$c\int_{x\in I}\eta\frac{\sqrt{1 + y'(x)^2}}{\sqrt{y(x)^2 + \left(6x^2 - \frac{g_2}{2}\right)^2}}\,dx = c\eps\eta\cdot\left[\frac{1 + y'(x_0)^2}{y(x_0)^2 + \left(6x_0^2 - \frac{g_2}{2}\right)^2}\right]^{\frac{1}{2}} + \cO\!\left(\eps^2\frac{d}{dx_0}\left[\frac{1 + y'(x_0)^2}{y(x_0)^2 + \left(6x_0^2 - \frac{g_2}{2}\right)^2}\right]^{\frac{1}{2}}\right),$$
  where $y(x) = \sqrt{4x^3 - g_2x - g_3}$ and $c$ is a normalization constant. The claim then follows from $y'(x) = \frac{12x^3 - g_2}{2y(x)}$ and straightforward calculation.
\end{proof}
\noindent It may be of interest to note that a less elementary but shorter proof of lemma \ref{diffdensity} can be formulated based on the observation that
$$\frac{dx}{y} = \frac{d(\wp(z))}{\wp'(z)} = \frac{\wp'(z)dz}{\wp'(z)} = dz.$$
~\\
The question of how quickly the set of multiples $nP$ will converge to the limiting density has been studied extensively in the theory of Diophantine approximation. See \cite{kn74} chapter 2, section 3 for an overview. For a point $P$ and an open set $U \subseteq \R^2$, let $r: \Z_{>0} \to \R$ be the function which satisfies
$$\#\{|k| < n \,:\, kP \in U\} = \rho n + r(n),$$
where $\rho$ is the natural density of multiples of $P$ which lie in the set $U$, as given by corollary \ref{density}. Then, for general $P$, it is not possible to give a better bound on $r(n)$ than $o(n)$, but for all but a set of points $P$ of measure $0$, we have $r(n) = \cO(n^{\frac{1}{2} + \eps})$ for every $\eps > 0$.\\
\section{Spacing}\label{spacingsection}
We can also study the statistics of the distances between the points $nP$ and $nP + Q$ for any fixed $Q$ in $E(\R)$. The raw moments of the distribution of distances diverge as more and more multiples of a fixed point $P$ are taken, as described in corollary \ref{spacingmomentslower}, and an upper bound for their growth in the number of multiples taken is given in corollary \ref{spacingmomentasymptotic}. We can, however, still find a distribution for these differences, as done in corollary \ref{spacing}.
\begin{corollary}\label{spacingmomentslower}
  For any points $P$ and $Q$ in $E(\R)$ and any positive integer $r$, the limits
  $$\lim_{n\to\infty}\frac{1}{n}\sum_{k=1}^n\Big(x(kP+Q) - x(kP)\Big)^r \quad\quad\text{and}\quad\quad \lim_{n\to\infty}\frac{1}{n}\sum_{k=1}^n\Big(y(kP+Q) - y(kP)\Big)^r$$
  diverge.
\end{corollary}
\begin{proof} Suppose these limits did converge. Let $z_P$ and $z_Q$ be the preimages of $P$ and $Q$ under the parameterization $z \mapsto (\wp(z),\wp'(z))$. Let $\omega_1$ and $\omega_2$ be the periods of $E/\C$, chosen such that $\omega_1 \in \R_{>0}$ and $\text{Im}(\omega_2) > 0$. Let $\Lambda$ be the lattice in $\C$ with basis $\langle \omega_1, \omega_2\rangle$. Define $I_P \subset \C/\Lambda$ as follows:
  \begin{align*}
    I_P := \begin{cases}[0,\omega_1], &\text{Im}(z_P) = 0 \mod\Lambda,\\
      [0,\omega_1]\cup \left([0,\omega_1]+\frac{\omega_2}{2}\right), &\text{Im}(z_P)=\text{Im}\!\left(\frac{\omega_2}{2}\right) \mod\Lambda,\end{cases}
  \end{align*}
  where $[0,\omega_1]$ denotes the interval of real numbers. Then theorem \ref{dist} implies that these limits would be equal to
  $$\frac{1}{\mu(I_P)}\int_{I_P}\Big(\wp(z+z_Q) - \wp(z)\Big)^r \,dz\quad\quad\text{and}\quad\quad \frac{1}{\mu(I_P)}\int_{I_P}\Big(\wp'(z+z_Q) - \wp'(z)\Big)^r \,dz,$$
  but both of these diverge because $\wp(z)$ has poles of order $2$ at $0$ and $\omega_1$.
\end{proof}
\begin{corollary}\label{spacingmomentasymptotic}
  Fix a point $Q$ in $E(\R)$ and a positive integer $r$. Then
  $$\sum_{k=1}^n\Big(x(kP+Q) - x(kP)\Big)^r \ll n^{2r+1+\eps}$$
  and
  $$\sum_{k=1}^n\Big(y(kP+Q) - y(kP)\Big)^r \ll n^{3r+1+\eps}$$
  for all points $P$ in $E(\R)$ except for a set of measure $0$, and all $\eps > 0$.
\end{corollary}
\begin{proof} Apply theorem \ref{typical} to $\psi(n) = n^{1+\eps}$.\end{proof}
~\\
Using theorem \ref{dist} we can conclude immediately that, for any $d \in \R$ and any $\eps > 0$,
$$\lim_{n\to\infty}\frac{1}{2n}\#\left\{|k| < n \,:\, \left|x(nP + Q) - x(nP) - d\right| < \eps\right\} = \frac{1}{\mu(I_P)}\mu\left(\left\{z\in I_P \,:\, \left|\wp(z+z_Q) - \wp(z) - d\right| < \eps\right\}\right).$$
Here we are using the notation from the proof of corollary \ref{spacingmomentslower} above. However, it is possible to write down the distribution of the spacings of these coordinates while avoiding making reference to the function $\wp(z)$. This is the content of corollary \ref{spacing}, which we now prove.
\begin{proof}[Proof of corollary \ref{spacing}.] Given an elliptic curve $E: y^2 = x^3 + ax + b$, a fixed point $Q = (x_Q,y_Q) \in E(\R)$, and a point $P = (x_P,y_P) \in E(\R)$ different from $Q$, we can compute directly using the chord and tangent law for addition on $E$ that
  $$x(P+Q) - x(P) = \left(\frac{y_P - y_Q}{x_P - x_Q}\right)^2 - 2x_P - x_Q.$$
  Fix $d \in \R$ and $\eps > 0$. We now wish to find the set of points $P$ for which $x(P+Q) - x(P) \in (d - \eps, d + \eps) \subset \R$. Substituting $y_P = \pm\sqrt{x_P^3 + ax_P + b}$, the condition we're interested in becomes
  $$\left|\,\left(\frac{\pm\sqrt{x_P^3 + ax_P + b} - y_Q}{x_P - x_Q}\right)^2 - 2x_P - x_Q - d\,\right| < \eps.$$
  Define
  $$F_{\pm,Q}(x) := \left(\frac{\pm\sqrt{x^3 + ax + b} - y_Q}{x - x_Q}\right)^2 - 2x - x_Q.$$
  Let $x^+_1,\dots,x^+_{k^+}$ and $x^-_1,\dots,x^-_{k^-}$ denote the real numbers which solve the equation $F_{\pm,Q}(x) = d$. Then, by considering the Taylor series expansion of $F_{\pm,Q}(x)$ around $x^\pm_i$ for $i = 1,\dots,k^\pm$, we see that whenever
  $$x^\pm_i - \frac{\eps}{F'_{\pm,Q}(x^\pm_i)} + \cO(\eps^2) \,\,< x <\,\, x^\pm_i + \frac{\eps}{F'_{\pm,Q}(x^\pm_i)} + \cO(\eps^2)$$
  we will have $|F_{\pm,Q}(x) - d| < \eps$.\\
  \\
  For any fixed $P_0 \in E(\R)$ of infinite order, lemma \ref{diffdensity} allows us to find the natural density of integers $n$ for which $x(nP_0)$ satisfies the condition $|F_{+,Q}(x(nP_0)) - d| < \eps$ or $|F_{-,Q}(x(nP_0)) - d| < \eps$. Define
  $$\rho(x) := \frac{1}{\sqrt{x^3 + ax + b}}.$$
  Then the values $x(nP_0)$ will have a distribution proportional to $\eta\cdot\rho(x)$, where $\eta = 1$ if $x$ is the $x$-coordinate of a point in the unbounded component of $E(\R)$ and $P_0$ is on the unbounded component of $E(\R)$, $\eta = \frac{1}{2}$ if $P_0$ is on the bounded component of $E(\R)$, and $\eta = 0$ otherwise. Hence, for fixed $\eps$, the natural density of integers $n$ for which $x(nP + Q) - x(nP) \in (d - \eps, d + \eps)$, as a function of $d$, is proportional to
  $$\sum_{i = 1}^{k^+}\!{\vphantom{\sum}}^*\frac{\rho(x^+_i)}{F'_{+,Q}(x^+_i)} + \cO\!\left(\eps\!\left((x^+_i)^3 + ax^+_i + b\right)^{-\frac{3}{2}}\right) + \sum_{i = 1}^{k^-}\!{\vphantom{\sum}}^*\frac{\rho(x^-_i)}{F'_{-,Q}(x^-_i)} + \cO\!\left(\eps\!\left((x^-_i)^3 + ax^-_i + b\right)^{-\frac{3}{2}}\right),$$
  where $\sum^*$ indicates that, if $P$ is on the unbounded component of $E(\R)$, then the sum omits the $i$ for which $x^\pm_i$ is not the $x$-coordinate of any point on the unbounded component of $E(\R)$.
\end{proof}
~\\
\noindent Informally, we can view the distribution $f(d)$ from theorem \ref{spacing} as
$$f(d) = \sum\rho\!\left(F_{+,Q}^{-1}(d)\right)\left(F_{+,Q}^{-1}\right)'\!\!(d) + \sum\rho\!\left(F_{-,Q}^{-1}(d)\right)\left(F_{-,Q}^{-1}\right)'\!\!(d),$$
where the sums are taken over the $k^\pm$ ``reasonable choices'' of the pair of values $\left(F_{\pm,Q}^{-1}(d),\,\left(F_{\pm,Q}^{-1}\right)'\!\!(d)\right)$.\\
~\\
\section{An equation of Bremner and Macleod}\label{bmsection}
In \cite{bm14}, Bremner and Macleod give positive integer solutions $a,b,c$ to the equation
\begin{equation}\tag{\ref{fruit}}
  \frac{a}{b+c}+\frac{b}{a+c}+\frac{c}{a+b} = N,
\end{equation}
where $N$ is an integer. Bremner and Macleod show that solutions to this equation are in bijection with rational points on the elliptic curve $E_N: Y^2 = X^3 + (4N^2 + 12N - 3)X^2 + 32(N+3)X$ with $X$-coordinate satisfying either
\begin{equation}\label{bmcond1}
  \frac{3 - 12N - 4N^2 - (2N+5)\sqrt{4N^2 + 4N - 15}}{2} < X < -2(N+3)\left(N + \sqrt{N^2 - 4}\right)
\end{equation}
or
\begin{equation}\label{bmcond2}
  -2(N+3)\left(N - \sqrt{N^2 - 4}\right) < X < -4\frac{N+3}{N+2}.
\end{equation}
Theorem \ref{dist} implies that for any $P \in E_N(\Q)$ of infinite order on the bounded connected component of $E_N(\R)$, the point $nP$ will correspond to a positive integer solution to (\ref{fruit}) for a certain specific proportion of integers $n$ in the sense of natural density. Writing down what this specific proportion is for general $N$ can be done using corollary \ref{density}. We first establish some notation.\\
\\
For brevity, define
$$A := 4N^2 + 12N - 3 \quad\quad\text{and}\quad\quad B := 32(N+3).$$
The curve $E_N$ is isomorphic to the curve $E'_N: y^2 = x^3 + ax + b$, where
$$a := B - \frac{A^2}{3} \quad\quad\text{and}\quad\quad b := \frac{2A^3}{27} - \frac{AB}{3},$$
via the transformation $y = Y$, $x = X + \frac{A}{3}$.\\
\\
By lemma \ref{diffdensity}, the function
$$\rho(x) := \frac{1}{\sqrt{x^3+ax+b}}$$
has the property that, for any $P \in E'_N(\Q)$, the distribution of $x(nP)$ will be proportional to $\eta\cdot \rho(x)$, where $\eta$ is $0$ if $x$ is the $x$-coordinate of a point in the bounded component of $E'_N(\R)$ and $P$ is in the unbounded component of $E'_N(\R)$, and $1$ otherwise.\\
\\
For brevity, define
\begin{align*}
  &x_{1,\text{left}} := \frac{3 - 12N - 4N^2 - (2N+5)\sqrt{4N^2 + 4N - 15}}{2} + \frac{A}{3}, &&x_{1,\text{right}} := -2(N+3)(N + \sqrt{N^2 - 4}) + \frac{A}{3},\\
  &x_{2,\text{left}} := -2(N+3)(N - \sqrt{N^2 - 4}) + \frac{A}{3}, &&x_{2,\text{right}} := -4\frac{N+3}{N+2} + \frac{A}{3}.
\end{align*}
These are the left and right edges of the intervals that $x$ can lie in if $X$ is to satisfy either condition (\ref{bmcond1}) or condition (\ref{bmcond2}). As explained in \cite{bm14}, all points on $E'_N$ which will yield a solution to (\ref{fruit}) are on the bounded component of $E'_N$. The $X$-coordinates of the points in the bounded component of $E_N(\R)$ form the interval $\left[\frac{-A-\sqrt{A^2 - 4B}}{2},\frac{-A+\sqrt{A^2 - 4B}}{2}\right]$, so the natural density of integers $n$ for which $nP$ solves (\ref{fruit}) is
$$\frac{\int_{x_{1,\text{left}}}^{x_{1,\text{right}}} \frac{dx}{\sqrt{x^3+ax+b}} + \int_{x_{2,\text{left}}}^{x_{2,\text{right}}} \frac{dx}{\sqrt{x^3+ax+b}}}{2\int_{\frac{-A - \sqrt{A^2-4B}}{2} + \frac{A}{3}}^{\frac{-A + \sqrt{A^2-4B}}{2} + \frac{A}{3}} \frac{dx}{\sqrt{x^3+ax+b}}}.$$
Noting that the integrand is the holomorphic differential $\omega$ mentioned in the introduction allows us to write this density more succinctly as
$$\frac{\int_{x_{1,\text{left}}}^{x_{1,\text{right}}} \omega + \int_{x_{2,\text{left}}}^{x_{2,\text{right}}} \omega}{2\omega_1}.$$
For $N = 4$, for example, this proportion is approximately $0.068$, while for $N = 38$ the proportion is approximately $0.003$.\\
\\
Bremner and Macleod consider the possibility that multiples of a generator $P$ of $E_N(\Q)$ could be uniformly distributed (with respect to the Lebesgue measure) on the bounded component of $E_N(\R)$. Corollary \ref{density} gives the exact distribution, and this distribution happens to be far from uniform in this case. The bounded component of $E_N$ has diameter $\cO(N^2)$, and, as $N$ goes to infinity, the section of the bounded component corresponding to condition (\ref{bmcond1}) has arclength $8N^2 + o(N^2)$ and the sections corresponding to (\ref{bmcond2}) have arclength totaling $32 + o(1)$. However, numerical experimentation suggests that
$$\int_{x_{1,\text{left}}}^{x_{1,\text{right}}} \omega = \frac{1}{2}\int_{x_{2,\text{left}}}^{x_{2,\text{right}}} \omega$$
for all $N$.
\\
\bibliographystyle{plain}
\bibliography{realpointsbib}{}

\end{document}